%% file: ribbon.tex
\documentclass{amsart}

\usepackage{graphicx}
\usepackage{morefloats,bm,caption,amsbsy,enumerate,amsmath,amsthm,
amssymb,mathtools,amsfonts,multirow,verbatim,tikz,tikz-cd,thmtools,thm-restate}
\usepackage{enumitem}
\usepackage[alphabetic]{amsrefs}
\usepackage{chngcntr}

\usepackage[hidelinks,colorlinks=true,linkcolor=blue, citecolor=black,linktocpage=true]{hyperref}
\usetikzlibrary{matrix,arrows,decorations.pathmorphing}
\usepackage[top=1.25in, bottom=1in, left=1.25in, right=1.25in,marginpar=1in]{geometry}
\usepackage{url}
\counterwithin{figure}{section}

\newtheorem{thm}{Theorem}[section]

\newtheorem{conj}[thm]{Conjecture}

\newtheorem{lem}[thm]{Lemma}

\theoremstyle{definition}
\newtheorem{define}[thm]{Definition}

\theoremstyle{remark}

\newcommand{\ve}[1]{\boldsymbol{\mathbf{#1}}}
\newcommand{\R}{\mathbb{R}}

\newcommand{\Z}{\mathbb{Z}}

\newcommand{\Q}{\mathbb{Q}}

\renewcommand{\d}{\partial}
\renewcommand{\subset}{\subseteq}

\renewcommand{\tilde}{\widetilde}

\newcommand{\iso}{\cong}

\DeclareMathOperator{\gr}{{gr}}

\DeclareMathOperator{\id}{{id}}

\DeclareMathOperator{\Int}{{int}}

\DeclareMathOperator{\rank}{{rank}}

\DeclareMathOperator{\Spin}{{Spin}}

\DeclareMathOperator{\Sym}{{Sym}}

\newcommand{\bF}{\mathbb{F}}

\newcommand{\bK}{\mathbb{K}}
\newcommand{\bL}{\mathbb{L}}

\newcommand{\bO}{\mathbb{O}}

\newcommand{\bT}{\mathbb{T}}

\newcommand{\cA}{\mathcal{A}}

\newcommand{\cC}{\mathcal{C}}
\newcommand{\cD}{\mathcal{D}}

\newcommand{\cF}{\mathcal{F}}

\newcommand{\cM}{\mathcal{M}}

\newcommand{\cS}{\mathcal{S}}

\newcommand{\frs}{\mathfrak{s}}
\newcommand{\frt}{\mathfrak{t}}

\newcommand{\cCFL}{\mathcal{C\!F\!L}}

\newcommand{\CFK}{\mathit{CFK}}
\newcommand{\HFK}{\mathit{HFK}}

\newcommand{\CFL}{\mathit{CFL}}
\newcommand{\HFL}{\mathit{HFL}}
\newcommand{\HFKh}{\widehat{\HFK}}
\newcommand{\HFLh}{\widehat{\HFL}}

\newcommand{\PD}{\mathit{PD}}

\newcommand{\xs}{\ve{x}}
\newcommand{\ys}{\ve{y}}
\newcommand{\zs}{\ve{z}}
\newcommand{\ws}{\ve{w}}

\newcommand{\as}{\ve{\alpha}}
\newcommand{\bs}{\ve{\beta}}

\renewcommand{\a}{\alpha}
\renewcommand{\b}{\beta}

\usepackage{leftidx}

\title{Knot Floer homology obstructs ribbon concordance}
\author{Ian Zemke}
\address{Department of Mathematics\\Princeton University\\  Princeton, NJ 08544, USA}
\email{izemke@math.princeton.edu}
\thanks{This research was supported by NSF grant DMS-1703685}

\begin{document}
\maketitle
	\begin{abstract}
		We prove that the map on knot Floer homology induced by a ribbon concordance is injective. As a consequence, we prove that the Seifert genus is monotonic under ribbon concordance. We also generalize a theorem of Gabai about the super-additivity of the Seifert genus under band connected sum. Our result gives evidence for a conjecture of Gordon that ribbon concordance is a partial order on the set of knots.
		\end{abstract}
\section{Introduction}

If $K_0$ and $K_1$ are knots in $S^3$, a \emph{concordance} from $K_0$ to $K_1$ is a smoothly embedded annulus in $[0,1]\times S^3$ with boundary $-\{0\}\times K_0\cup \{1\}\times K_1$. A \emph{ribbon concordance} is a concordance $C$ with only index 0 and 1 critical points. A \emph{slice knot} is one which is concordant to the unknot (or equivalently, one which bounds a smoothly embedded disk in $B^4$).  A \emph{ribbon knot} is one which admits a ribbon concordance from the unknot to $K$. 

A major open problem in topology is the \emph{slice-ribbon conjecture}, which asks whether every slice knot is ribbon. In this paper, we discuss the related problem of determining when two concordant knots are ribbon concordant.

Some classical results about ribbon concordances are due to Gordon \cite{Gordon}. Suppose $C$ is a ribbon concordance from $K_0$ to $K_1$.  Write $\pi_1(K_i)$ for the fundamental group of the complement of $K_i$ in $S^3$, and $\pi_1(C)$ for the fundamental group of the complement of $C$ in $[0,1]\times S^3$. Gordon \cite{Gordon}*{Lemma~3.1} proved that
\[
\pi_1 (K_0)\to \pi_1(C)\quad \text{is injective \qquad and}\qquad \pi_1( K_1)\to \pi_1(C) \quad \text{is surjective.}
\]
 Note that in Gordon's terminology, such a concordance goes `from' $K_1$ `to' $K_0$, though this is the opposite of the cobordism orientation, which is more convenient for our present paper.

In this paper, we show that knot Floer homology gives an obstruction to ribbon concordance.

\begin{figure}[ht!]
	\centering
	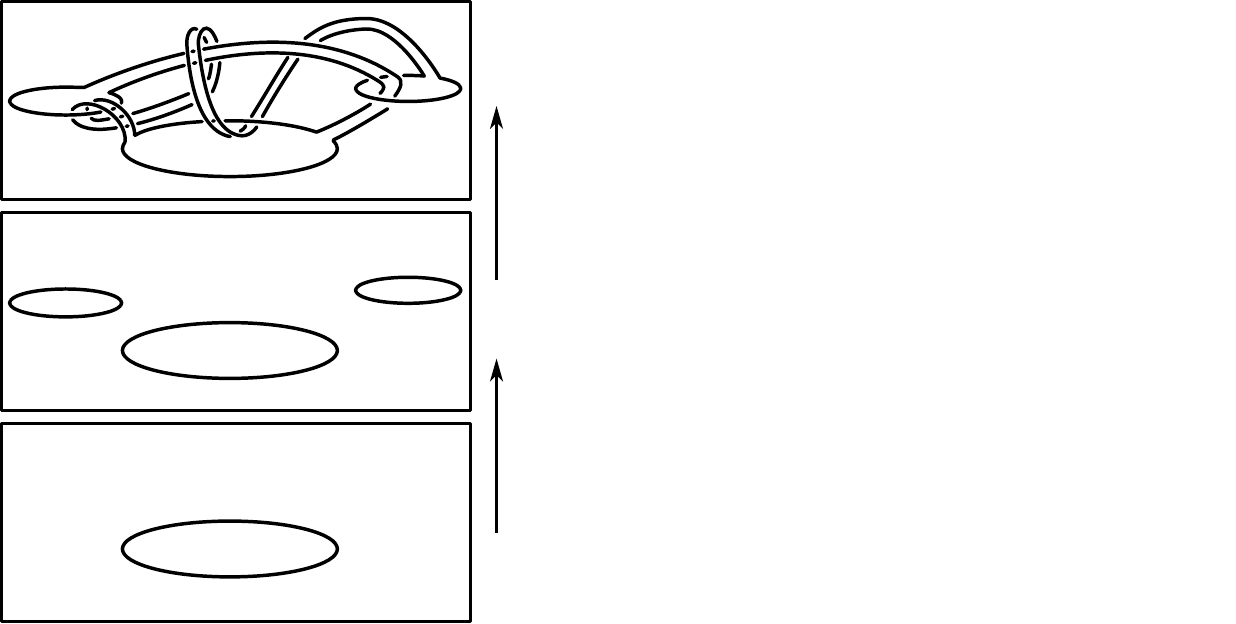
	\caption{\textbf{A ribbon concordance from $K_0$ to $K_1$.}}\label{fig::1}
\end{figure}

\subsection{Knot Floer homology and ribbon concordances}

 If $K\subset S^3$ is a knot, there is a bigraded $\bF_2$ vector space
\begin{equation}
\HFKh(K)=\bigoplus_{i,j\in \Z}\HFKh_i(K,j),\label{eq:hatHFK}
\end{equation}
constructed independently by Ozsv\'{a}th and Szab\'{o} \cite{OSKnots}, and Rasmussen \cite{RasmussenKnots}. The subscript $i$ in Equation~\eqref{eq:hatHFK} denotes the Maslov grading, and $j$ the Alexander grading.

If $C$ is a concordance from $K_0$ to $K_1$, Juh\'{a}sz and Marengon \cite{JMConcordance} construct a grading preserving cobordism map
\[
F_{C}\colon  \widehat{\HFK}(K_0)\to \widehat{\HFK}(K_1),
\]
which is well defined up to two graded automorphisms of $\HFKh(K_0)$ and $\HFKh(K_1)$. The ambiguity corresponds to a choice of decoration on $C$; See Section~\ref{sec:background-on-HFL} for further details.

The concordance maps are based on a more general construction of cobordism maps on link Floer homology due to Juh\'{a}sz \cite{JCob}. We will also make use of an alternate description given by the author \cite{ZemCFLTQFT}, which extends to the minus and infinity flavors of link Floer homology.

Our main theorem is the following:

\begin{thm}\label{thm:1} If $C$ is a ribbon concordance from $K_0$ to $K_1$, then the map
\[
F_{C}\colon \widehat{\HFK}(K_0)\to \widehat{\HFK}(K_1)
\]
is an injection.
\end{thm}

The idea of our argument is very simple. Let $C'$ denote the concordance from $K_1$ to $K_0$ obtained by turning $C$ upside down.  We will show that
\[
F_{C'}\circ F_{C}=\id_{\HFKh(K_0)},
\]
which immediately implies Theorem~\ref{thm:1}.

We will in fact show that a version of Theorem~\ref{thm:1} holds for the full knot Floer complex, $\CFK^\infty(K)$, which contains more information than $\HFKh(K)$; See Theorem~\ref{thm:1infty} and Section~\ref{sec:full-knot-Floer-complex}.

\subsection{Topological applications}
An immediate corollary of Theorem~\ref{thm:1} is the following:

\begin{thm}\label{thm:rank} If there is a ribbon concordance from $K_0$ to $K_1$, then for each $i$ and $j$
\[
\rank_{\bF_2} \HFKh_i(K_0,j)\le \rank_{\bF_2} \HFKh_i(K_1,j).
\]
\end{thm}

If $K$ is a knot, let $d(K)$ denote the degree of the Alexander polynomial of $K$. Gordon \cite{Gordon}*{Lemma~3.4} showed that if there is a ribbon concordance from $K_0$ to $K_1$, then
\begin{equation}
d(K_0)\le d(K_1).\label{eq:ineq-alexander-deg}
\end{equation}

Ozsv\'{a}th and Szab\'{o} \cite{OSgenusbounds}*{Theorem~1.2} proved that knot Floer homology detects the Seifert genus:
\begin{equation}
g_3(K)=\max\left\{ i: \HFKh(K,i)\neq \{0\}\right\}.\label{eq:Seifert-detection}
\end{equation}
Juh\'{a}sz gave an alternate argument using surface decompositions and sutured manifolds \cite{JuhaszSurfaceDecomp}*{Theorem~1.5}.

Analogous to Gordon's result in Equation~\eqref{eq:ineq-alexander-deg}, an immediate consequence of Theorem~\ref{thm:1} and Equation~\eqref{eq:Seifert-detection} is the following:

\begin{thm}\label{thm:monotone-seifert-genus} If there is a ribbon concordance from $K_0$ to $K_1$, then
\[
g_3(K_0)\le g_3(K_1).
\]
\end{thm}


In a different direction, Gordon makes the following conjecture:

\begin{conj}[\cite{Gordon}] 
Ribbon concordance is a partial ordering, i.e. if there is a ribbon concordance from $K_0$ to $K_1$, and also a ribbon concordance from $K_1$ to $K_0$, then $K_0=K_1$.
\end{conj}

Our Theorem~\ref{thm:1} gives the following immediate corollary, which supports Gordon's conjecture:
\begin{thm}\label{thm:isoHFK} If there is a ribbon concordance from $K_0$ to $K_1$, and also a ribbon concordance from $K_1$ to $K_0$, then
\[
\HFKh(K_0)\iso \HFKh(K_1),
\]
as bigraded vector spaces over $\bF_2$.
\end{thm}

In fact, we will show that something stronger is true: the $\Z\oplus\Z$-filtered full knot Floer complexes $\CFK^\infty(K_0)$ and $\CFK^\infty(K_1)$ are isomorphic; See Theorem~\ref{thm:1infty}.

A caveat to Theorem~\ref{thm:isoHFK} is that although $\HFKh$ detects the unknot \cite{OSgenusbounds}, as well as trefoils and the figure-eight knot \cite{GhigginiFibered} \cite{NiFibered}, there are infinite families of knots which have the same knot Floer homology \cite{HeddenWatsonBotanyGeography}*{Theorem~1}.

\subsection{Seifert genus of band connected sums}

If $K_1,\dots, K_n$ are knots in $S^3$, a \emph{band connected sum} of $K_1,\dots, K_n$ is a knot $L$ obtained by connecting $K_1,\dots, K_n$ together with $n-1$ bands. The ordinary connected sum is an example of a band connected sum, but in general, band connected sums will be more complicated.

Gabai \cite{GabaiBand} proved that if $L$ is a band connected sum of $K_1$ and $K_2$, then
\begin{equation}
g_3(L)\ge g_3(K_1)+g_3(K_2).\label{eq:subadditivity}
\end{equation}
If equality holds in Equation~\eqref{eq:subadditivity}, then Gabai also proved that $L=K_1\# K_2$.

Note that the band connected sum of three or more knots is not in general an iterated band connected sum of pairs of knots. Gabai's proof does not obviously extend to the case of three or more summands. We prove the following:

\begin{thm}\label{thm:bandsumseifert}
If a knot $L$ is a band connected sum of knots $K_1,\dots, K_n$, then
\begin{equation}
g_3(L)\ge g_3(K_1)+\cdots +g_3(K_n).\label{eq:super-additivity}
\end{equation}

\end{thm}
\begin{proof}
Miyazaki \cite{MiyazakiBandSumRibbon} gave an elegant manipulation which shows that if $L$ is the band connected sum of $K_1,\dots, K_n$, then there is a ribbon concordance from $K_1\# \cdots \# K_n$ to $L$.  Hence, Equation~\eqref{eq:super-additivity} follows immediately from our Theorem~\ref{thm:monotone-seifert-genus}.
\end{proof}

Some comments are in order. Miyazaki \cite{MiyazakiFiberedBand} recently proved super-additivity of the Seifert genus under the assumption that $L$ is fibered, in fact showing that $K_1,\dots, K_n$ must all be fibered as well. Miyazaki combines results of Gordon \cite{Gordon}, Silver \cite{SilverRibbon} and Kochloukova~\cite{Kochloukova} to show that if equality holds in Equation~\eqref{eq:super-additivity} and $L$ is fibered, then $L=K_1\# \cdots \# K_n$.

\subsection{Extension to the full knot Floer complex}
\label{sec:intro-L-spaces}
Ozsv\'{a}th and Szab\'{o} \cite{OSKnots} defined a more general version of knot Floer homology,  called the \emph{full knot Floer complex}, denoted $\CFK^\infty(K)$. The object $\CFK^\infty(K)$ is a $\Z\oplus \Z$-filtered chain complex over the ring $\bF_2[U,U^{-1}]$.

The present author gave a functorial construction of cobordism maps for the full knot Floer complex  \cite{ZemCFLTQFT}.  As a generalization to Theorem~\ref{thm:1}, we will show the following:

\begin{thm}\label{thm:1infty} If $C$ is a ribbon concordance from $K_0$ to $K_1$, then there is a filtered, $\bF_2[U,U^{-1}]$-equivariant chain map $\Pi\colon \CFK^\infty(K_1)\to \CFK^\infty(K_0)$ such that
\[
\Pi\circ F_{C}\colon \CFK^\infty(K_0)\to \CFK^\infty(K_0)
\]
is chain homotopic to the identity, via a filtered, $\bF_2[U,U^{-1}]$-equivariant chain homotopy. Indeed if $C'$ is the concordance obtained by reversing $C$, then we can take $\Pi=F_{C'}$.
\end{thm}

\subsection{Further commentary}
\label{sec:furthercommentary}

A recent paper of Miyazaki~\cite{MiyazakiFiberedBand} points out that work of Silver~\cite{SilverRibbon} and Kochloukova~\cite{Kochloukova} together imply that if there is a ribbon concordance from $K_0$ to $K_1$, and $K_1$ is fibered, then $K_0$ is also fibered. Silver reduced the problem to a conjecture of Rapaport \cite{RapaportKnotLike} about knot-like groups, which Kochloukova proved. Consequently, if there is a ribbon concordance from $K_0$ to $K_1$ and $K_1$ is fibered, and further $K_0$ and $K_1$ have the same Seifert genus, then \cite{Gordon}*{Lemma~3.4} implies they must be isotopic. Note that our Theorem~\ref{thm:1} gives an alternate proof of this fact which avoids Kochloukova's result, by using Ni's theorem that knot Floer homology detects fibered knots \cite{NiFibered} together with \cite{Gordon}*{Lemma~3.4}.

Finally, we remark that a major open problem in symplectic topology is determining whether every Lagrangian concordance between Legendrian knots in $S^3$ is \emph{decomposable}; See~\cite{ChantraineNonCollarable}*{Definition~1.4}, \cite{Ekholm-Honda-Kalman}*{Section~6}. Decomposable Lagrangian cobordisms are products of elementary cobordisms corresponding to Legendrian Reidemeister moves, saddles and births. In particular, decomposable Lagrangians are ribbon. One strategy for proving that a given Lagrangian concordance is not decomposable might be to show that it is not ribbon via our Theorem~\ref{thm:rank} (or more ambitiously Theorem~\ref{thm:1}, if one could explicitly compute the map). Unfortunately the only candidates the author is aware of are satellites of decomposable Legendrian concordances \cite{CornwallNgSivek}*{Section~2.2}, and hence are ribbon.

\subsection{Acknowledgments}

I would like to thank Jen Hom, Tye Lidman, Jeffrey Meier and Maggie Miller for helpful correspondences. This problem was posed by Sucharit Sarkar at the CMO-BIRS Conference \emph{Thirty Years of Floer Theory for 3-Manifolds}  in Oaxaca, Mexico; See Problem 26 of the problem list \url{https://www.birs.ca/cmo-workshops/2017/17w5011/report17w5011.pdf}.

\section{Background on knot and link Floer homology}
\label{sec:background-on-HFL}

Knot Floer homology is an invariant of knots discovered independently by Ozsv\'{a}th and Szab\'{o} \cite{OSKnots}, as well as Rasmussen \cite{RasmussenKnots}. Ozsv\'{a}th and Szab\'{o} \cite{OSLinks} constructed a generalization, called link Floer homology, associated to links in 3-manifolds. In this section, we present background material about knot and link Floer homology.

\begin{define}
A \emph{multi-based link} $\bL = (L,\ws,\zs)$ in a 3-manifold $Y$ is an oriented link
$L \subset Y$, together with two disjoint collections of basepoints $\ws,\zs \subset L$ such that the following hold:
\begin{enumerate}
\item Each component of $L$ has at least two basepoints.
\item The basepoints alternate between $\ws$ and $\zs$, as one traverses $L$.
\end{enumerate}
\end{define}

To a multi-based link $\bL$ in $Y$, the link Floer homology group
\[
\widehat{\HFL}(Y,\bL)
\]
is a vector space over $\bF_2$. If $\bK=(K,w,z)$ is a doubly based knot in $S^3$, the group $\HFLh(S^3,\bK)$ coincides with the  knot Floer homology group $\HFKh(K)$. When $Y\neq S^3$, the group $\HFLh(Y,\bL)$ decomposes along $\Spin^c$ structures  as
\[
\HFLh(Y,\bL)=\bigoplus_{\frs\in \Spin^c(Y)} \HFLh(Y,\bL,\frs),
\]
 as we outline below.

We briefly describe the construction of link Floer homology. One starts with a Heegaard diagram $(\Sigma,\as,\bs,\ws,\zs)$ for $\bL$; see \cite{OSLinks}*{Section~3.5} for the definition of a Heegaard diagram of a multi-based link. Write
\[
\as=\{\alpha_1,\dots, \alpha_n\}\qquad \text{and} \qquad \bs=\{\beta_1,\dots, \beta_n\},
\]
where $n=g(\Sigma)+|\ws|-1=g(\Sigma)+|\zs|-1$, and consider the two half-dimensional tori
\[
\bT_{\a}=\alpha_1\times \cdots \times \alpha_n, \qquad \text{and}\qquad  \bT_{\b}=\beta_1\times \cdots \times \beta_n
\]
inside of the symmetric product $\Sym^n(\Sigma)$. 

There is a map $\frs_{\ws}\colon \bT_{\a}\cap \bT_{\b}\to \Spin^c(Y)$ defined by Ozsv\'{a}th and Szab\'{o} \cite{OSDisks}*{Section~2.6}. As a module over $\bF_2$, the chain complex $\widehat{\CFL}(Y,\bL,\frs)$
is freely generated by the intersection points $\xs\in \bT_{\a}\cap \bT_{\b}$ which satisfy $\frs_{\ws}(\xs)=\frs$.
The differential $\d$ on $\widehat{\CFL}(Y,\bL,\frs)$ is defined by counting holomorphic disks in $\Sym^n(\Sigma)$ with zero multiplicity on $\ws$ and $\zs$:
\begin{equation}
\d \xs=\sum_{\ys\in \bT_{\a}\cap \bT_{\b}} \sum_{\substack{\phi\in \pi_2(\xs,\ys)\\
\mu(\phi)=1\\
n_{\ws}(\phi)=n_{\zs}(\phi)=0}} \# (\cM(\phi)/\R)\cdot \ys. \label{eq:differential-hat}
\end{equation}

The definition of link Floer homology can be extended to disconnected manifolds via a tensor product, as long as each component of the 3-manifold contains a component of the link. By convention, we set
\[
\HFLh(\emptyset):=\bF_2.
\]

Functorial cobordism maps for the hat flavor of link Floer homology were constructed by Juh\'{a}sz \cite{JCob}. Juh\'{a}sz's construction made use of the Honda--Kazez--Mati\'{c} gluing map \cite{HKMTQFT}. The present author \cite{ZemCFLTQFT} gave an alternate construction of link cobordism maps in terms of elementary cobordisms. The construction is independent of the contact-geometric construction of Honda, Kazez and Mati\'{c}. In a joint work with Juh\'{a}sz, the author showed that the two constructions yield the same cobordism maps \cite{JuhaszZemkeContactHandles}*{Theorem~1.4}.

Juh\'{a}sz's link Floer TQFT uses the following notion of a decorated link cobordism between two multi-based links:
\begin{define}
Let $Y_0$ and $Y_1$ be 3-manifolds containing multi-based links $\bL_0 = (L_0, \ws_0, \zs_0)$
and $\bL_1 = (L_1, \ws_1, \zs_1)$, respectively. A \emph{decorated link cobordism} from
$(Y_0, \bL_0)$ to $(Y_1, \bL_1)$ is a pair $(W,\cF)=(W,(\Sigma,\cA))$, satisfying the following:
\begin{enumerate}
\item $W$ is an oriented cobordism from $Y_0$ to $Y_1$.
\item $\Sigma$ is a properly embedded, oriented surface in $W$ with $\d \Sigma = -L_0 \cup L_1$.
\item $\cA$ is a properly embedded 1-manifold in $\Sigma$
that divides $\Sigma$ into two subsurfaces $\Sigma_{\ws}$ and $\Sigma_{\zs}$
that meet along $\cA$, such that $\ws_0,\ws_1 \subset \Sigma_{\ws}$
and $\zs_0,\zs_1 \subset \Sigma_{\zs}$.
\end{enumerate}
\end{define}

Using the constructions from \cite{JCob} and \cite{ZemCFLTQFT}, if $\frs\in \Spin^c(W)$, there is a functorial cobordism map
\[
F_{W,\cF,\frs}\colon \HFLh(Y_0,\bL_0,\frs|_{Y_0})\to \HFLh(Y_1,\bL_1,\frs|_{Y_1}).
\]

When $\Spin^c(W)$ contains only one element, $\frs$, we write simply 
\[
F_{W,\cF}:=F_{W,\cF,\frs}.
\]

To a concordance $C$ from $K_0$ to $K_1$, we decorate $K_0$ and $K_1$ each with a pair of basepoints, and obtain a decorated link cobordism $([0,1]\times S^3,\cC)$ by decorating $C$ with two parallel dividing arcs, both going from $K_0$ to $K_1$. This configuration is studied in \cite{JMConcordance}. The choice of such dividing arcs is not canonical, since we can always apply a Dehn twist along a homotopically nontrivial curve in $C$. Hence, if $C$ is an undecorated concordance, the induced cobordism map is only well defined up to the automorphisms of knot Floer homology induced by the diffeomorphisms which twist $K_0$ or $K_1$ in one full twist. Note that composition with a grading preserving automorphism does not affect the statement of Theorem~\ref{thm:1}. The basepoint moving automorphism map has been studied  by Sarkar \cite{SarkarMovingBasepoints} and by the author \cite{ZemQuasi}.

Next, we discuss gradings. If $\bL$ is a null-homologous link in $Y$ (i.e. the total class of $\bL$ vanishes in $H_1(Y;\Z)$) and $\frs$ is a torsion $\Spin^c$ structure on $Y$, Ozsv\'{a}th and Szab\'{o} construct two gradings on link Floer homology: the Maslov and Alexander gradings. In the framework of our TQFT, it is convenient to repackage these two gradings into three, which satisfy a linear dependency. There are two Maslov gradings $\gr_{\ws}$ and $\gr_{\zs}$, as well as an Alexander grading $A$, which satisfy
\[
A=\frac{1}{2}(\gr_{\ws}-\gr_{\zs}).
\]
The Maslov grading described by Ozsv\'{a}th and Szab\'{o} is equal to $\gr_{\ws}$, in our notation.

The cobordism maps are graded, and the author \cite{ZemAbsoluteGradings}*{Theorem~1.4} showed that if $\frs|_{Y_0}$ and $\frs|_{Y_1}$ are torsion, and $\bL_0$ and $\bL_1$ are null-homologous, then
\begin{equation}
\gr_{\ws}(F_{W,\cF,\frs}(\xs))-\gr_{\ws}(\xs)=\frac{c_1(\frs)^2-2\chi(W)-3\sigma(W)}{4}+\tilde{\chi}(\Sigma_{\ws})\label{eq:gradingchange1}
\end{equation}
and
\begin{equation}
\gr_{\zs}(F_{W,\cF,\frs}(\xs))-\gr_{\zs}(\xs)=\frac{c_1(\frs-\PD[\Sigma])^2-2\chi(W)-3\sigma(W)}{4}+\tilde{\chi}(\Sigma_{\zs}),\label{eq:gradingchange2}
\end{equation}
where 
\[
\tilde{\chi}(\Sigma_{\ws}):=\chi(\Sigma_{\ws})-\frac{1}{2}(|\ws_0|+|\ws_1|)
\]
and $\tilde{\chi}(\Sigma_{\zs})$ is defined similarly. Special cases of the above grading formulas were independently proven by Juh\'{a}sz and Marengon \cite{JMComputeCobordismMaps}, when $W=[0,1]\times S^3$.

A final property that we will need concerns the behavior of the cobordism maps for 2-knots in $S^4$:

\begin{lem}\label{lem:maps-for-2-spheres} Suppose $(S^4, \cS)\colon \emptyset\to \emptyset$ is a decorated link cobordism such that $\cS$ is a smooth 2-knot decorated with a single dividing curve. The induced map
\[
F_{S^4,\cS}\colon \bF_2\to \bF_2
\]
is an isomorphism.
\end{lem}
Lemma~\ref{lem:maps-for-2-spheres} follows  from \cite{JMConcordance}*{Theorem~1.2}. Alternatively, we can view it as a consequence of a more general formula for the behavior of the cobordism maps applied to closed surfaces, due to the author \cite{ZemAbsoluteGradings}*{Theorem~1.8}.

\section{Proof of Theorem~\ref{thm:1}}
Having reviewed the necessary background, we now prove our main result:

\begin{proof}[Proof of Theorem~\ref{thm:1}] Suppose $C$ is a ribbon concordance from $K_0$ to $K_1$. Let $\cC$ denote $C$, decorated with two parallel dividing arcs running from $K_0$ to $K_1$. 

Consider the concordance $C'$ from  $K_1$ to  $K_0$ obtained by turning around and reversing the orientation of $C$. Let $\cC'$ denote the concordance $C'$ with the decorations induced by $\cC$. Write $C'\circ C$ for the concordance from $K_0$ to itself obtained by concatenating $C$ and $C'$, and write $\cC'\circ \cC$ for $C'\circ C$ decorated with the arcs from $\cC$ and $\cC'$.

We claim that
\begin{equation}
F_{[0,1]\times S^3,\cC'} \circ F_{[0,1]\times S^3, \cC}=F_{[0,1]\times S^3, \cC'\circ \cC}=\id_{\HFKh(K_0)}. \label{eq:main-equation}
\end{equation}
Note that Equation~\eqref{eq:main-equation} immediately implies Theorem~\ref{thm:1}. The first equality in Equation~\eqref{eq:main-equation} follows from the composition law for link cobordisms, so it remains to prove the second.

The concordance $C'\circ C$ will not in general  be isotopic to the product 
\[
[0,1]\times K_0\subset [0,1]\times S^3.
\]
 Nonetheless, the link Floer TQFT cannot tell the difference, as we now precisely describe.

Pick a movie presentation of $C$, with the following form:
\begin{enumerate}[label=($M$-\arabic*), ref=$M$-\arabic*]
\item $n$ births, adding unknots $U_1,\dots, U_n$, each contained in a small ball, disjoint from $K_0$ and the other $U_i$.
\item $n$ saddles, for bands $B_1,\dots, B_n$, such that $B_i$ connects $U_i$ to $K_0$. 
\item \label{movie:3} A final isotopy taking the band surgered knot $(K_0\cup U_1\cup \cdots \cup  U_n)(B_1,\dots, B_n)$ to $K_1$.
\end{enumerate}

Such a movie can be obtained by taking the concordance $C$ (which by assumption has only index 0 and 1 critical points), and moving the index 0 critical points below the index 1 critical points. \emph{A-priori} the bands induced by the index 1 critical points may not have one end on $K_0$ and one end on one of $U_1,\dots, U_n$. However, after a sequence of band slides, it is easy to arrange for this configuration.

The concordance $C'\circ C$ can be given a movie by concatenating the above movie with its reverse. In this movie for $C'\circ C$, we run the isotopy from~\eqref{movie:3} forward in the $C$-portion of the movie, and then immediately run it backwards in the $C'$-portion. Consequently, we can delete the two adjacent levels corresponding to isotopy in the movie for $C'\circ C$, and obtain the following movie for $C'\circ C$:

\begin{enumerate}[label=($M'$-\arabic*), ref=$M'$-\arabic*]
\item $n$ births, adding $U_1,\dots, U_n$.
\item $n$ saddles, for the bands $B_1,\dots, B_n$.
\item $n$ saddles, for bands $B_1',\dots, B_n'$ which are dual to $B_1,\dots, B_n$.
\item $n$ deaths, deleting $U_1,\dots, U_n$.
\end{enumerate}

For each $i$, the band $B_i$, together with its dual $B_i'$, determines a tube (i.e. an annulus), for which we write $T_i$. The births and deaths determine 
2-spheres, $S_1,\dots, S_n$. Although the 2-spheres $S_i$ are individually unknotted, the tubes $T_i$ may link the spheres $S_i$ in a very complicated manner.

Consequently, we can view the concordance $C'\circ C$ as being obtained by taking the identity concordance $[0,1]\times K_0$, and tubing in the spheres $S_1,\dots, S_n$ using the tubes $T_1,\dots, T_n$.

We view the tubes as the boundaries of 3-dimensional 1-handles $h_1,\dots, h_n$ embedded in $[0,1]\times S^3$ which join the surface $[0,1] \times K_0$ to the spheres $S_1,\dots, S_n$. The 1-handle $h_i$ intersects $S_i$ in a disk $D_i$, and $h_i$ intersects $[0,1]\times K_0$ in a disk $D_{i}^0$. Let us write $D_i'$ for the disk $D_i':=S_i\setminus \Int( D_i)$.

 The concordance $C'\circ C$ is equal to the union
\begin{equation}
C'\circ C= \bigg([0,1]\times K_0\setminus (D_{1}^0\cup \cdots \cup D_{n}^0)\bigg)\cup (T_1\cup \cdots \cup T_n)\cup (D_1'\cup \cdots\cup  D_n'). \label{eq:simple-decomp-of-C'UC}
\end{equation}

If we replace the expression $D_1'\cup \cdots \cup D_n'$ appearing in Equation~\eqref{eq:simple-decomp-of-C'UC} with $D_1\cup \cdots \cup D_n$, we obtain a surface which is isotopic to the identity concordance $[0,1]\times K_0$; See Figure~\ref{fig::2}.

We now claim that replacing $D_i'$ with $D_i$ does not change the cobordism map. Let $N(S_i)$ denote a regular neighborhood of $S_i$, and let $Y_i$ denote the boundary of $N(S_i)$. Note that
\[ 
N(S_i)\iso D^2\times S_i\qquad \text{and} \qquad Y_i\iso S^1\times S_i.
\]

The concordance $C'\circ C$ intersects $Y_i$ in an unknot $O_i$ (equal to the intersection of $T_i$ with $Y_i$). We isotope the dividing set on $\cC'\circ \cC$ so that it intersects $O_i$ in two points, and intersects $D_i'$ in a single arc. Let $\cD_i'$ denote $D_i'$ decorated with this dividing arc, and let $\cD_i$ denote $D_i$ decorated with a single dividing arc. Let $\bO_i$ denote the unknot $O_i\subset Y_i$, decorated with two basepoints, compatibly with the dividing arcs.

We now claim that
\begin{equation}
F_{D^2\times S_i, \cD_i',\frt_0}=F_{D^2\times S_i, \cD_i,\frt_0}\label{eq:crux}
\end{equation}
as maps from $\HFLh(\emptyset)\iso \bF_2$ to $\HFLh(Y_i,\bO_i,\frt_0')$, where $\frt_0\in \Spin^c(D^2\times S_i)$ denotes the $\Spin^c$ structure which evaluates trivially on $S_i$, and $\frt_0'$ denotes its restriction to $Y_i\iso S^1\times S_i$.

Note that Equation~\eqref{eq:crux}, together with the composition law for link cobordisms \cite{ZemCFLTQFT}*{Theorem~B}, implies the main claim.

 We prove Equation~\eqref{eq:crux} in the subsequent Lemma~\ref{lem:replace-2-sphere}.
\end{proof}

\begin{figure}[ht!]
	\centering
	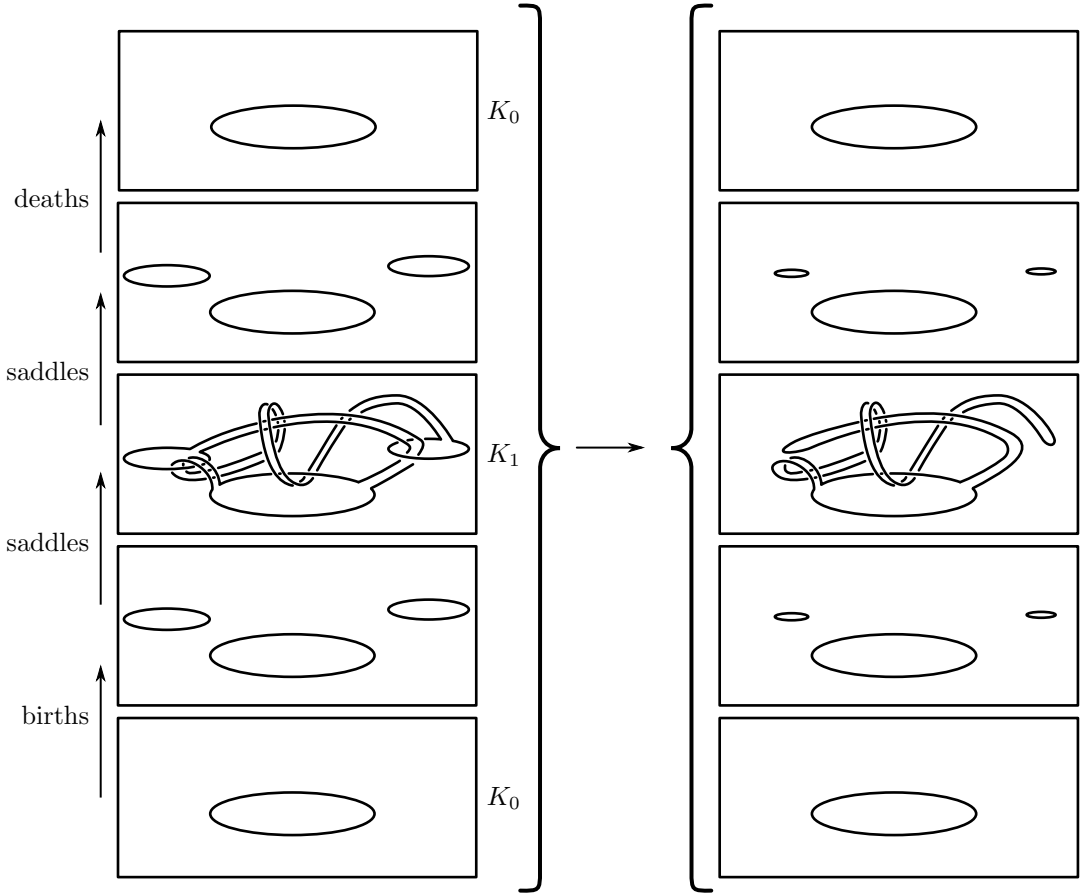
	\caption{\textbf{The modification of $C'\circ C$ from the proof of Theorem~\ref{thm:1}.} On the left is $C'\circ C$. On the right is the concordance obtained by replacing the disks $D_i'$ with $D_i$. The concordance on the right is isotopic to $ [0,1] \times K_0$.}\label{fig::2}
\end{figure}

\begin{lem}\label{lem:replace-2-sphere} Suppose $D$ and $D'$ are two smooth, properly embedded disks in $W:=D^2\times S^2$ which intersect $S^1\times S^2$ in an unknot $O$. Let $\bO$ denote $O$ decorated with two basepoints, and let $\cD$ and $\cD'$ denote $D$ and $D'$ decorated with a single dividing arc, compatibly with the basepoints of $\bO$. Let $\frt_0$ denote the unique $\Spin^c$ structure on $W$ whose Chern class evaluates trivially on $\{0\}\times S^2$, and let $\frt_0'$ denote its restriction to $S^1\times S^2$.
Then
\begin{equation}
F_{W, \cD,\frt_0}=F_{W,\cD',\frt_0},  \label{eq:cobordism-maps-coincide}
\end{equation}
as maps from $\bF_2\iso \HFKh(\emptyset)$ to $\HFKh(S^1\times S^2, \bO,\frt_0')$.
\end{lem}
\begin{proof} 
A Heegaard diagram for $(S^1\times S^2,\bO)$ is shown in Figure~\ref{fig::3}.
\begin{figure}[ht!]
	\centering
	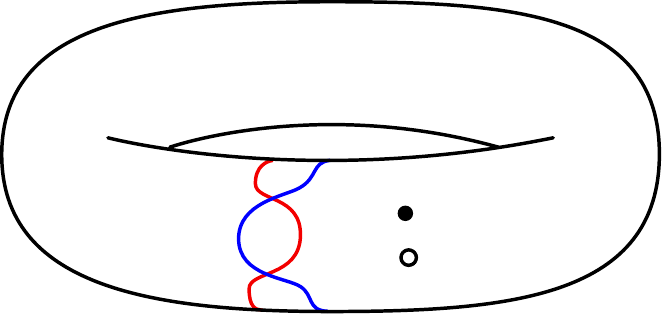
	\caption{\textbf{A diagram for the unknot $\bO$ in $S^1\times S^2$.} The two intersection points both represent the torsion $\Spin^c$ structure.}\label{fig::3}
\end{figure}

Since $\bO$ is a doubly based unknot, the gradings $\gr_{\ws}$ and $\gr_{\zs}$ coincide on the generators of $\HFKh(S^1\times S^2, \bO,\frt_0')$. Indeed both $\gr_{\ws}$ and $\gr_{\zs}$ are defined on intersection points using the same formula, since the basepoints are immediately adjacent. Let us write $\gr$ for the common Maslov grading on $\HFKh(S^1\times S^2,\bO,\frt_0')$.

In particular
\[
\HFKh(S^1\times S^2,\bO,\frt_0')\iso (\bF_2)_{-\frac{1}{2}}\oplus (\bF_2)_{\frac{1}{2}},
\]
where $(\bF_2)_p$ denotes a rank 1 summand of $\bF_2$, concentrated in $\gr$-grading $p\in \Q$.

We make two claims:
\begin{enumerate}[label=($c$-\arabic*), ref=$c$-\arabic*]
\item\label{claim:C1} Both $F_{W,\cD,\frt_0}(1)$ and $F_{W,\cD',\frt_0}(1)$ have $\gr$-grading $-\tfrac{1}{2}$.
\item\label{claim:C2} $F_{W,\cD,\frt_0}(1)$ and $F_{W,\cD',\frt_0}(1)$ are both non-zero.
\end{enumerate}

Claim~\eqref{claim:C1} follows from the grading change formulas in Equations~\eqref{eq:gradingchange1} and~\eqref{eq:gradingchange2} (both formulas give the same answer).

Claim~\eqref{claim:C2} is proven as follows: Let $(W_0,\cD_0)\colon (S^1\times S^2,\bO)\to \emptyset$ denote a decorated link cobordism where $W_0$ is a 3-handle cobordism followed by a 4-handle, and $\cD_0$ is a smooth disk, decorated with a single dividing arc. Note that $W_0\cup W$ is diffeomorphic to $S^4$. Write $\cS$ and $\cS'$ for the 2-spheres $\cD_0\cup \cD$ and $\cD_0\cup \cD'$, respectively. By the composition law,
\begin{equation}
F_{S^4,\cS}=F_{W_0,\cD_0}\circ F_{W,\cD,\frt_0}.\label{eq:decompose-2-sphere}
\end{equation}
However $F_{S^4,\cS}\colon \bF_2\to \bF_2$ is an isomorphism by Lemma~\ref{lem:maps-for-2-spheres}. Consequently $F_{W,\cD,\frt_0}(1)$ must be nonzero in light of Equation~\eqref{eq:decompose-2-sphere}. The same argument shows $F_{W,\cD',\frt_0}(1)$ is nonzero, so Claim~\eqref{claim:C2} holds.

Noting that $\HFKh(S^1\times S^2, \bO,\frt_0')$ has rank 1 in $\gr$-grading $-\tfrac{1}{2}$, Claims~\eqref{claim:C1} and \eqref{claim:C2} imply that $F_{W,\cD,\frt_0}=F_{W,\cD',\frt_0}$, completing the proof.
\end{proof}

\section{Extension to the full knot Floer complex}
\label{sec:full-knot-Floer-complex}

In this section, we prove Theorem~\ref{thm:1infty}. The argument is only notationally harder than the one we gave for Theorem~\ref{thm:1}. We first recall the definition of the relevant version of the full knot Floer complex and the cobordism maps from \cite{ZemCFLTQFT}, and state some basic properties.

If $\bL$ is a multi-based link in $Y$, we let $\cCFL^-(Y,\bL,\frs)$ denote the module which is freely generated over the two variable polynomial ring $\bF_2[u,v]$ by intersection points $\xs\in \bT_{\a}\cap \bT_{\b}$ with $\frs_{\ws}(\xs)=\frs$. Adapting Equation~\eqref{eq:differential-hat}, we equip $\cCFL^-(Y,\bL,\frs)$ with the differential
\[
\d \xs=\sum_{\ys\in \bT_{\a}\cap \bT_{\b}} \sum_{\substack{\phi\in \pi_2(\xs,\ys)\\
\mu(\phi)=1}} \# (\cM(\phi)/\R)\cdot u^{n_{\ws}(\phi)} v^{n_{\zs}(\phi)}\cdot \ys,
\]
where $n_{\ws}(\phi)$ and $n_{\zs}(\phi)$ denote the total multiplicity of the class $\phi$ on the basepoints $\ws$ and $\zs$.

Note that $\widehat{\CFL}(Y,\bL,\frs)$ is obtained from $\cCFL^-(Y,\bL,\frs)$ by setting $u=v=0$. A complex $\cCFL^\infty(Y,\bL,\frs)$ is defined by formally inverting $u$ and $v$ in the module $\cCFL^-(Y,\bL,\frs)$.

The gradings $\gr_{\ws}$, $\gr_{\zs}$ and $A$ described in Section~\ref{sec:background-on-HFL} all have incarnations on the minus and infinity flavors. On intersection points, their definitions coincide with the gradings on $\widehat{\CFL}(Y,\bL,\frs)$. They are extended to $\cCFL^-(Y,\bL,\frs)$ by defining $u$ to have $(\gr_{\ws},\gr_{\zs})$-bigrading $(-2,0)$, and $v$ to have $(\gr_{\ws},\gr_{\zs})$-bigrading $(0,-2)$. The Alexander grading  satisfies $A=\tfrac{1}{2}(\gr_{\ws}-\gr_{\zs})$. 

Let $\bK$ denote a knot $K$ decorated with two basepoints. The full knot Floer complex $\CFK^\infty(K)$, as defined in \cite{OSKnots}, is equal to the subcomplex of $\cCFL^-(S^3,\bK)$ in Alexander grading zero. The group $\CFK^\infty(K)$ has an action of $\bF_2[U,U^{-1}]$, by having $U$ act by the product $uv$. The actions of $u$ and $v$  are not individually well defined on $\CFK^\infty(K)$, since they have non-zero Alexander grading.

To a decorated link cobordism $(W,\cF)\colon (Y_0,\bL_0)\to (Y_1,\bL_1)$, equipped with a $\Spin^c$ structure $\frs\in \Spin^c(W)$, the author \cite{ZemCFLTQFT} constructs  functorial cobordism maps
\[
F_{W,\cF,\frs}\colon \cCFL^-(Y_0,\bL_0,\frs|_{Y_0})\to \cCFL^-(Y_1,\bL_1,\frs|_{Y_1}).
\]

If $\bL_0$ and $\bL_1$ are both null-homologous and $\frs|_{Y_0}$ and $\frs|_{Y_1}$ are torsion (so the gradings are defined) then the grading formulas from Equations~\eqref{eq:gradingchange1} and~\eqref{eq:gradingchange2} hold \cite{ZemAbsoluteGradings}*{Theorem~1.4}.

Next, we recall the form of the maps for closed surfaces in $S^4$, as computed by the author \cite{ZemAbsoluteGradings}*{Theorem~1.8}:

\begin{lem}\label{lem:more-general-2-knot-comp} Suppose that $\cS=(\Sigma,\cA)$ is a closed, decorated surface in $S^4$ such that $\Sigma\setminus \cA$ has two connected components, $\Sigma_{\ws}$ and $\Sigma_{\zs}$. Then the cobordism map
\[
F_{S^4,\cS}\colon \bF_2[u,v]\to \bF_2[u,v]
\]
is equal to the map
\[
1\mapsto u^{g(\Sigma_{\ws})} v^{g(\Sigma_{\zs})}.
\]
\end{lem}

In particular, when $\cS$ is a 2-knot, the cobordism map is the identity.

We are now equipped to prove Theorem~\ref{thm:1infty}:
\begin{proof}[Proof of Theorem~\ref{thm:1infty}]
Suppose $C$ is a ribbon concordance from $K_0$ to $K_1$, and let $\cC$ and $\cC'$ be decorations of $C$ and $C'$, as in the proof of Theorem~\ref{thm:1}. The proof of the present theorem amounts to showing that
\begin{equation}
F_{[0,1]\times S^3,\cC'} \circ F_{[0,1]\times S^3, \cC}\simeq F_{[0,1]\times S^3, \cC'\circ \cC}\simeq \id_{\cCFL^-(K_0)}.\label{eq:maps-compose-to-identity}
\end{equation}

The proof of Equation~\eqref{eq:maps-compose-to-identity} follows the proof of Theorem~\ref{thm:1} verbatim until Lemma~\ref{lem:replace-2-sphere}. We claim that Lemma~\ref{lem:replace-2-sphere} holds with $\widehat{\CFL}(S^1\times S^2,\bO,\frt'_0)$ replaced by $\cCFL^-(S^1\times S^2,\bO,\frt'_0)$.

Note that $\gr_{\ws}$ and $\gr_{\zs}$ do not coincide on $\cCFL^-(S^1\times S^2,\bO,\frt_0')$, since $u$ and $v$ have non-zero $(\gr_{\ws},\gr_{\zs})$-bigrading. We claim that the following analogs of Claims~\eqref{claim:C1} and~\eqref{claim:C2} hold:
\begin{enumerate}[label=($c'$-\arabic*), ref=$c'$-\arabic*]
\item\label{claim:C1'} Both $F_{W,\cD,\frt_0}(1)$ and $F_{W,\cD',\frt_0}(1)$ have $(\gr_{\ws},\gr_{\zs})$-bigrading $(-\tfrac{1}{2},-\tfrac{1}{2})$ in $\cCFL^-(S^1\times S^2,\bO,\frt_0')$.
\item\label{claim:C2'} $F_{W,\cD,\frt_0}(1)$ and $F_{W,\cD',\frt_0}(1)$ are both non-zero.
\end{enumerate}

Claim~\eqref{claim:C1'} follows from Equations~\eqref{eq:gradingchange1} and~\eqref{eq:gradingchange2}, as before.

 Claim~\eqref{claim:C2'} is proven similarly to Claim~\eqref{claim:C2}, except using Lemma~\ref{lem:more-general-2-knot-comp} for the maps induced by 2-knots.

Finally, we note that
\[
\cCFL^-(S^1\times S^2, \bO,\frt_0')\iso \bigg((\bF_2)_{(-\frac{1}{2},-\frac{1}{2})}\oplus (\bF_2)_{(\frac{1}{2},\frac{1}{2})}\bigg)\otimes_{\bF_2} \bF_2[u,v],
\]
with vanishing differential (See Figure~\ref{fig::3} for a diagram). In the above equation, $(\bF_2)_{(p,q)}$ denotes a summand of $\bF_2$ concentrated in $(\gr_{\ws},\gr_{\zs})$-bigrading $(p,q)$. The homogeneous subset in $(\gr_{\ws},\gr_{\zs})$-bigrading $(-\tfrac{1}{2},-\tfrac{1}{2})$ has rank 1, so Claims~\eqref{claim:C1'} and ~\eqref{claim:C2'} imply that 
\[
F_{W,\cD,\frt_0}(1)=F_{W,\cD',\frt_0}(1)
\]
as elements of $\cCFL^-(S^1\times S^2,\bO,\frt_0')$. Using the composition law, the proof is complete.

\end{proof}

\bibliographystyle{custom}
\bibliography{biblio}

\end{document}

%% file: fig1.pdf_tex
\begingroup%
  \makeatletter%
  \providecommand\color[2][]{%
    \errmessage{(Inkscape) Color is used for the text in Inkscape, but the package 'color.sty' is not loaded}%
    \renewcommand\color[2][]{}%
  }%
  \providecommand\transparent[1]{%
    \errmessage{(Inkscape) Transparency is used (non-zero) for the text in Inkscape, but the package 'transparent.sty' is not loaded}%
    \renewcommand\transparent[1]{}%
  }%
  \providecommand\rotatebox[2]{#2}%
  \newcommand*\fsize{\dimexpr\f@size pt\relax}%
  \newcommand*\lineheight[1]{\fontsize{\fsize}{#1\fsize}\selectfont}%
  \ifx\svgwidth\undefined%
    \setlength{\unitlength}{355.72211383bp}%
    \ifx\svgscale\undefined%
      \relax%
    \else%
      \setlength{\unitlength}{\unitlength * \real{\svgscale}}%
    \fi%
  \else%
    \setlength{\unitlength}{\svgwidth}%
  \fi%
  \global\let\svgwidth\undefined%
  \global\let\svgscale\undefined%
  \makeatother%
  \begin{picture}(1,0.50460733)%
    \lineheight{1}%
    \setlength\tabcolsep{0pt}%
    \put(0,0){\includegraphics[width=\unitlength,page=1]{fig1.pdf}}%
    \put(0.41265209,0.12629005){\color[rgb]{0,0,0}\makebox(0,0)[lt]{\lineheight{1.25}\smash{\begin{tabular}[t]{l}births\end{tabular}}}}%
    \put(0.41398323,0.32690521){\color[rgb]{0,0,0}\makebox(0,0)[lt]{\lineheight{1.25}\smash{\begin{tabular}[t]{l}saddles\end{tabular}}}}%
    \put(0,0){\includegraphics[width=\unitlength,page=2]{fig1.pdf}}%
    \put(0.95643215,0.05396438){\color[rgb]{0,0,0}\makebox(0,0)[t]{\lineheight{1.25}\smash{\begin{tabular}[t]{c}$K_0$\end{tabular}}}}%
    \put(0.95643215,0.44150081){\color[rgb]{0,0,0}\makebox(0,0)[t]{\lineheight{1.25}\smash{\begin{tabular}[t]{c}$K_1$\end{tabular}}}}%
    \put(0,0){\includegraphics[width=\unitlength,page=3]{fig1.pdf}}%
  \end{picture}%
\endgroup%

%% file: fig2.pdf_tex
\begingroup%
  \makeatletter%
  \providecommand\color[2][]{%
    \errmessage{(Inkscape) Color is used for the text in Inkscape, but the package 'color.sty' is not loaded}%
    \renewcommand\color[2][]{}%
  }%
  \providecommand\transparent[1]{%
    \errmessage{(Inkscape) Transparency is used (non-zero) for the text in Inkscape, but the package 'transparent.sty' is not loaded}%
    \renewcommand\transparent[1]{}%
  }%
  \providecommand\rotatebox[2]{#2}%
  \newcommand*\fsize{\dimexpr\f@size pt\relax}%
  \newcommand*\lineheight[1]{\fontsize{\fsize}{#1\fsize}\selectfont}%
  \ifx\svgwidth\undefined%
    \setlength{\unitlength}{410.16984675bp}%
    \ifx\svgscale\undefined%
      \relax%
    \else%
      \setlength{\unitlength}{\unitlength * \real{\svgscale}}%
    \fi%
  \else%
    \setlength{\unitlength}{\svgwidth}%
  \fi%
  \global\let\svgwidth\undefined%
  \global\let\svgscale\undefined%
  \makeatother%
  \begin{picture}(1,0.81744531)%
    \lineheight{1}%
    \setlength\tabcolsep{0pt}%
    \put(0,0){\includegraphics[width=\unitlength,page=1]{fig2.pdf}}%
    \put(0.09040176,0.15480582){\color[rgb]{0,0,0}\makebox(0,0)[rt]{\lineheight{1.25}\smash{\begin{tabular}[t]{r}births\end{tabular}}}}%
    \put(0.09040176,0.31416252){\color[rgb]{0,0,0}\makebox(0,0)[rt]{\lineheight{1.25}\smash{\begin{tabular}[t]{r}saddles\end{tabular}}}}%
    \put(0,0){\includegraphics[width=\unitlength,page=2]{fig2.pdf}}%
    \put(0.09040176,0.47183308){\color[rgb]{0,0,0}\makebox(0,0)[rt]{\lineheight{1.25}\smash{\begin{tabular}[t]{r}saddles\end{tabular}}}}%
    \put(0,0){\includegraphics[width=\unitlength,page=3]{fig2.pdf}}%
    \put(0.09040176,0.63024548){\color[rgb]{0,0,0}\makebox(0,0)[rt]{\lineheight{1.25}\smash{\begin{tabular}[t]{r}deaths\end{tabular}}}}%
    \put(0,0){\includegraphics[width=\unitlength,page=4]{fig2.pdf}}%
    \put(0.45559177,0.0801953){\color[rgb]{0,0,0}\makebox(0,0)[lt]{\lineheight{1.25}\smash{\begin{tabular}[t]{l}$K_0$\end{tabular}}}}%
    \put(0.45559177,0.39593998){\color[rgb]{0,0,0}\makebox(0,0)[lt]{\lineheight{1.25}\smash{\begin{tabular}[t]{l}$K_1$\end{tabular}}}}%
    \put(0.45559177,0.71168452){\color[rgb]{0,0,0}\makebox(0,0)[lt]{\lineheight{1.25}\smash{\begin{tabular}[t]{l}$K_0$\end{tabular}}}}%
  \end{picture}%
\endgroup%

%% file: fig3.pdf_tex
\begingroup%
  \makeatletter%
  \providecommand\color[2][]{%
    \errmessage{(Inkscape) Color is used for the text in Inkscape, but the package 'color.sty' is not loaded}%
    \renewcommand\color[2][]{}%
  }%
  \providecommand\transparent[1]{%
    \errmessage{(Inkscape) Transparency is used (non-zero) for the text in Inkscape, but the package 'transparent.sty' is not loaded}%
    \renewcommand\transparent[1]{}%
  }%
  \providecommand\rotatebox[2]{#2}%
  \newcommand*\fsize{\dimexpr\f@size pt\relax}%
  \newcommand*\lineheight[1]{\fontsize{\fsize}{#1\fsize}\selectfont}%
  \ifx\svgwidth\undefined%
    \setlength{\unitlength}{190.40476015bp}%
    \ifx\svgscale\undefined%
      \relax%
    \else%
      \setlength{\unitlength}{\unitlength * \real{\svgscale}}%
    \fi%
  \else%
    \setlength{\unitlength}{\svgwidth}%
  \fi%
  \global\let\svgwidth\undefined%
  \global\let\svgscale\undefined%
  \makeatother%
  \begin{picture}(1,0.47384271)%
    \lineheight{1}%
    \setlength\tabcolsep{0pt}%
    \put(0,0){\includegraphics[width=\unitlength,page=1]{fig3.pdf}}%
    \put(0.47356825,0.09696139){\color[rgb]{1,0,0}\makebox(0,0)[lt]{\lineheight{1.25}\smash{\begin{tabular}[t]{l}$\alpha$\end{tabular}}}}%
    \put(0.34302947,0.09643823){\color[rgb]{0,0,1}\makebox(0,0)[rt]{\lineheight{1.25}\smash{\begin{tabular}[t]{r}$\beta$\end{tabular}}}}%
    \put(0.63675447,0.15956313){\color[rgb]{0,0,0}\makebox(0,0)[lt]{\lineheight{1.25}\smash{\begin{tabular}[t]{l}$w$\end{tabular}}}}%
    \put(0.64589854,0.09485137){\color[rgb]{0,0,0}\makebox(0,0)[lt]{\lineheight{1.25}\smash{\begin{tabular}[t]{l}$z$\end{tabular}}}}%
  \end{picture}%
\endgroup%